\documentclass[11pt]{amsart}
\usepackage{amsmath}
\usepackage{amssymb}
\usepackage{graphicx}
\usepackage[noadjust]{cite}
\newtheorem{theorem}{Theorem}[section]
\newtheorem{corollary}[theorem]{Corollary}
\newtheorem{lemma}[theorem]{Lemma}
\newtheorem{proposition}[theorem]{Proposition}
\theoremstyle{definition}
\newtheorem{definition}[theorem]{Definition}
\newtheorem{example}[theorem]{Example}
\newtheorem{remark}[theorem]{Remark}

\numberwithin{equation}{section}

\title[Solving nonlinear Volterra--Fredholm integral equations ]{A numerical method for solving nonlinear Volterra--Fredholm integral equations}

\author[N. T. Binh]{Ngo Thanh Binh}

\address[N. T. Binh]{Faculty of Fundamental Science, Namdinh University of Technology Education, Nam Dinh, Vietnam}
\email{{\tt ntbinhspktnd@gmail.com}}

\author[K. V. Ninh]{Khuat Van Ninh}
\address[K. V. Ninh]{Department of Mathematics, Hanoi Pedagogical University  2, 	Phuc Yen, Vietnam}
\email{\tt kvnkhoatoan@gmail.com}

\thanks{This research is funded by Vietnam National Foundation for Science and Technology Development (NAFOSTED) under grant number 101.92.2014.51.}



\begin{document}

\begin{abstract}
In this paper, we introduce a numerical method for  solving nonlinear Volterra--Fredholm integral equations. Our method consists of two steps. First, we define a discretized form of the integral equation  by  quadrature methods.  We then propose an iterative method, which is based on a hybrid of the method of contractive mapping and parameter continuation method, to solve  the perturbed system of nonlinear equations obtained from discretization of the considered problem. Finally,  an example is given to demonstrate the validity and applicability of our method.
\end{abstract}

\maketitle


\section{Introduction}

\indent Many problems which arise in mathematics, physics, biology, etc., lead to integral equations. Volterra--Fredholm integral equations play an important role in the theory of  integral equations.  Since Volterra--Fredholm integral equations are usually difficult to get their exact solution, therefore, many authors have worked on analytical methods and numerical methods for approximate solution of this kind of equations. For example,  Adomian decomposition method was used to solve Volterra--Fredholm integral equations in  \cite{RefBWaz}.  H. Brunner \cite{RefJBru} and  P. J. Kauthen \cite{RefJKau} have employed the Collocation method to solve  mixed Volterra--Fredholm integral equations. The numerical solutions of the nonlinear Volterra--Fredholm integral equations by using Homotopy perturbation method was introduced in \cite{RefJGha}. Y. Mahmoudi \cite{RefJMah} and S. Yal\c{c}inba\c{s} \cite{RefJYal} developed the Taylor polynomial solutions for the nonlinear Volterra--Fredholm integral equations.  In \cite{RefJOrd},  Y. Ordokhani  applied the rationalized Haar functions to solve nonlinear Volterra--Fredholm--Hammerstein integral equations. Recently,  M. Zarebnia \cite{RefJZar} used the sinc functions to solve the nonlinear Volterra--Fredholm integral equations. F. Mirzaee and A. A. Hoseini \cite{RefJMiz} obtained an approximate solution for the nonlinear Volterra--Fredholm integral equations by  the hybrid of block--pulse function and Taylor series (HBT).  An iterative scheme for extracting approximate solutions of the Volterra--Fredholm integral equations has been presented by A. H. Borzabadi and M. Heidari \cite{RefJBor}. J. Xie et al. \cite{RefJXie} developed  the numerical technique based on block--pulse functions (BPFs)  to approximate the solutions of nonlinear Volterra--Fredholm--Hammerstein integral equations in two--dimensional spaces.  \\
\indent In this paper, we intend to present a numerical method for approximating
the solution of nonlinear Volterra--Fredholm integral equation as follows
\begin{equation}\label{eq:1.1}
x(t) +  \int\limits_a^t {K_1(t,s,x(s)) ds} + \int\limits_a^b {K_2(t,s,x(s)) ds} = g(t),\;\,a \le t \le b,
\end{equation}
where $x(t)$ is an unknown function that will be determined, $g(t), K_i(t,s,x),$ $ i = 1,2$  are known functions  and $a, b$ are known constants. At first, we use one of  frequently used quadrature methods to reduce the equation \eqref{eq:1.1} into a  perturbed system of nonlinear equations. For further information on quadrature methods in this respect, see \cite{RefBCh1,RefBPrem,RefJSto}. 
Next, we propose an iterative method to solve the obtained perturbed system of nonlinear equations.  The method is based on a hybrid of the method of contractive mapping and parameter continuation method. Parameter continuation method was suggested and developed by S. N. Bernstein \cite{RefJB} and J. Schauder \cite{RefJLS}.  Later on, V. A. Trenoghin \cite{RefBTre1,RefJTre2,RefJTre3,RefJTre4} has developed a generalized variants of the parameter continuation method and used to prove the invertibility of nonlinear operators, which map a metric space or a weak metric space into a Banach space. Y. L. Gaponenko \cite{RefJG}  proposed and justifed the parameter continuation method for solving operator equations of the second kind with a Lipschitz - continuous and monotone operator, which operates in an arbitrary Banach space. K. V. Ninh \cite{RefJNi1, RefJNi2} has studied parameter continuation method for solving the operator equations of the second kind with a sum of two operators. In \cite{RefJVe2}, V. G. Vetekha presented the application of parameter continuation method to solving the boundary 
value problem for the ordinary differential equations of second order. Parameter continuation method has some advantages that encourage us to use it. Firstly,  the properties of contractions such as iteration,  error estimates are used to find approximate solutions and estimate the errors of approximate solutions. Furthermore, this method is very simple to apply and to make an algorithm.\\
\indent The paper is organized as follows. In Section \ref{sec:2}, the parameter continuation method for solving operator equations of the second kind is briefly presented. In this section, we recall some definitions and results that will be useful in the sequel. In Section \ref{sec:3}, we transform  the equation \eqref{eq:1.1} into a  perturbed system of nonlinear equations. Then we  discuss the existence and uniqueness of the solution of the obtained perturbed system of nonlinear equations and prove that its solution converges to the exact solution of the problem. We also study approximate solutions of the  perturbed system of nonlinear equations and their error estimates. An illustrative example is  given in Section \ref{sec:4} to  illustrate the efficiency of the introduced method.  Finally, Section \ref{sec:5} draws some conclusions from the paper.
\section{Parameter continuation method for solving operator equations of the second kind}
\label{sec:2}
In this section, we  recall some definitions and results which we will use in the sequel. For details, we refer to \cite{RefJG}.\\
\indent  Let $X$ be a Banach space and $A$ be a mapping, which operates in the space $X$. Consider the operator equation of the second kind 
\begin{equation}\label{eq:2.1}
x + A(x) = f.
\end{equation}
\begin{definition}[{\cite{RefJG}}]\label{def1}
	The mapping $A$, which operates in the Banach space $X$ is called monotone  if for any elements $x_1, x_2 \in X$ and any $\varepsilon>0$ the following inequality holds 
	\begin{equation}\label{eq:2.2}
	\| x_1 -x_2 + \varepsilon \left[A(x_1) - A(x_2) \right]  \|  \ge  \|x_1 - x_2 \|. 
	\end{equation}
\end{definition}
\begin{remark}[{\cite{RefJG}}]\label{rem1}
	If $X$ is Hilbert space then the condition of monotonicity \eqref{eq:2.2} is equivalent to the classical condition 
	\[\left\langle {A(x_1) - A(x_2),x_1 - x_2} \right\rangle  \ge 0,\;\forall x_1, x_2 \in X,\]
	where $\left\langle , \right\rangle $ is an inner product in the Hilbert space $X$.
\end{remark}
\begin{lemma}[\cite{RefJG}]\label{lem2.1}
	Assume that $A$ is a monotone mapping which operates in the Banach space $X$. Then for any elements $x_1, x_2 \in X$ and any positive numbers $\varepsilon_1, \varepsilon_2, 0 < \varepsilon_1 \le \varepsilon_2 \le 1$, the following inequality holds
	\[ \| x_1 - x_2 + \varepsilon_1 \left[A(x_1) - A(x_2) \right]  \| \le \| x_1 - x_2 + \varepsilon_2 \left[A(x_1) - A(x_2) \right]  \|.  \]
\end{lemma}
\indent The basic idea of  the parameter continuation method for solving operator equations of the second kind \eqref{eq:2.1} is as follows. Consider  a one-parametric family of equations
\begin{equation*}
x + \varepsilon A(x) = f, 0 \le \varepsilon \le 1,
\end{equation*}
which when $\varepsilon =0$ gives the trivial equation $x=f$ and when $\varepsilon =1$ gives the initial equation \eqref{eq:2.1}. Dividing $[0,1]$ into $N$ equal parts with $N$ is a natural number such that  $q = L \varepsilon_0 <1, \varepsilon_0 = \frac{1}{N}$, where $L$ is Lipschitz coefficient of the operator $A$. After  $N-1$ changes of variables
\begin{subequations}
	\begin{align}
	u &=x+ \varepsilon_0 A(x) \equiv G_1 (x), \label{eq:2.3a} \\ 
	v &= u +  \varepsilon_0 AG_1^{-1}(u) \equiv G_2(u), \label{eq:2.3b} \\
	& \; \ldots, \; \label{eq:2.3c} \\ 
	y &= \omega + \varepsilon_0 AG_1^{-1} \cdots G_{N-2}^{-1}(\omega) \equiv G_{N-1}(\omega),\label{eq:2.3d}
	\end{align}
\end{subequations}
we construct intermediate equations with contractive operators in new variables.  By virtue of the monotonicity and Lipschitz  continuity of the operator $A$, contraction coefficients of these contractive operators  equal  $q$. By shifting the parameter $\varepsilon$ step  by step $\varepsilon_0$  from $0$ to $1$  we can verify that the equation \eqref{eq:2.1}  has a unique solution.
\begin{theorem}[{\cite{RefJG}}]\label{thm2.1}
	Suppose that the mapping  $A$, which operates in the Banach space $X$ is Lipschitz - continuous and monotone. Then  the equation \eqref{eq:2.1} has a unique solution for any element $f \in X$. 
\end{theorem}
To find approximate solutions of the equation \eqref{eq:2.1}, Y. L. Gaponenko has  constructed the following iteration process 
\begin{equation}\label{eq:2.4}
x_{k+1} = \underbrace {-  \frac{1}{N}A(x_{k}) - \frac{1}{N}A(x_{l}) - \cdots -  \frac{1}{N}A(x_{p})}_{N\;terms}  + f,\; k, l, \ldots, p = 0, 1, 2, \ldots .
\end{equation}
The symbolic notation \eqref{eq:2.4} should be understood as the following  iteration processes, which consist of $N$ iteration processes 
\begin{subequations}
	\begin{align}
	x_{k+1}& = - \varepsilon_0 A( x_k) + u_l,\; k =0, 1, 2, \ldots,\label{eq:2.5a}\\
	u_{l+1}& = - \varepsilon_0 A G_1^{-1}(u_l)+ v_c,\; l=0,1, 2,  \ldots ,\label{eq:2.5b}\\
	& \; \ldots,\label{eq:2.5c}\\
	y_{p+1} & = - \varepsilon_0 A G_1^{-1} \cdots G_{N-1}^{-1}(y_p) + f, \; p=0, 1, 2, \ldots.\label{eq:2.5d}
	\end{align}
\end{subequations}
\indent For simplicity, assume that $A (0) = 0$ and the number of steps in each iteration scheme of the iteration process \eqref{eq:2.4} is the same and equals $n_0$. Denote $x(n_0,N) \equiv x_{n_0}$ as the approximate solutions of the equation \eqref{eq:2.1}, which is constructed by the iteration process \eqref{eq:2.4}.  In this case, Y. L. Gaponenko received the error estimations of approximate solutions of the equation \eqref{eq:2.1}, which are presented in the following theorem. 
\begin{theorem}[{\cite{RefJG}}]\label{thm2.2}
	Assume that the conditions of Theorem \ref{thm2.1} are satisfied. Then the sequence of approximate solutions $\{x(n_0,N)\},\; n_0 = 1, 2, \ldots$ constructed by iteration process \eqref{eq:2.4} converges to the exact solution $x^*$ of the equation \eqref{eq:2.1}. Moreover, the following estimates hold 
	\begin{equation}\label{eq:2.6}
	\|x(n_0,N) - x^* \| \le \frac{q^{n_0+1}}{1-q} \frac{e^{qN}-1}{e^q - 1} \|f\|, 
	\end{equation}
	where $L$ is Lipschitz coefficient of the operator $A, N$ is the smallest natural number such that $q = \frac{L}{N}<1, n_0 = 1, 2, \ldots$.
\end{theorem}
\section{Main results}
\label{sec:3}
\indent Supposing that the discussed integral equation \eqref{eq:1.1} has solution. We will consider  the  equation 
\eqref{eq:1.1}  under the following  assumptions:
\begin{enumerate}
		\item [(i)] $g(t) \in C^{\nu}[a,b], \nu \ge 2$;
		\item [(ii)] $K_1(t,s,x(s)),  K_2(t,s,x(s))$   are differentiable continuous functions up
		to order  $\nu$ on  $\Omega = [a,b] \times [a,b] \times \mathbb{R}$, where $\nu \ge 2$;
	\item [(iii)] $K_1(t,s,x)$ satisfies a Lipschitz condition of the type
	\[  \left| K_1(t,s,x) - K_1(t,s,\overline{x})    \right| \le | \psi (t,s)| \left|x-\overline{x}\right|,   \]
	for all $a \le t,s \le b$ and for all reals $x,\overline{x}$, where  $\int\limits_a^t {|\psi(t,s)|^2ds} \le Q^2(t)$ and $\int\limits_a^b{Q^2(t)dt} \le M^2 < + \infty$;
	\label{item:1}
	\item [(iv)] $K_2(t,s,x)$ satisfies a Lipschitz condition of the type
	\[  \left| K_2(t,s,x) - K_2(t,s,\overline{x})    \right| \le \left|\phi(t,s) \right|  \left|x-\overline{x}\right|,      \]
	for all $a \le t,s \le b$ and for all reals $x, \overline{x}$, where $ \int\limits_a^b \int\limits_a^b {\left|\phi(t,s) \right|^2dsdt}=L^2< + \infty$;
	\label{item:2}
	\item [(v)] $K_2(t,s,x)$ satisfies the condition
	\[  \int\limits_a^b{ \left\{\int\limits_a^b { \left[ K_2(t,s,x(s)) - K_2(t,s,\overline{x}(s))\right] ds}\right\} \left[x(t) - \overline{x}(t)\right]dt} > 0,\] 
	for all  $x(t), \overline{x}(t) \in C^{\nu}[a,b]$ with $x(t) \ne \overline{x}(t)$.
	\label{item:3}
\end{enumerate}

\indent At the beginning, we transform the  equation \eqref{eq:1.1} into a discretized  form. Let $\mathsf{\Pi} =  \left\{a = t_0, t_1, \ldots, t_{n-1}, t_n = b\right\}$ be an equidistant partition of $[a,b]$ where $h = t_{i+1} - t_i, \; i = 0, 1, \ldots, n-1$ is the discretization parameter of the partition.  Now, if $x^*(t)$ is an analytical solution of 
\eqref{eq:1.1}, then for the partition $\mathsf{\Pi}$ on $[a,b]$, we have
\begin{equation}\label{eq:3.1}
x^*(t_i) +   \int\limits_a^{t_i} {K_1(t_i,s, x^*(s))ds} + \int\limits_a^b {K_2(t_i,s, x^*(s))ds} = g(t_i),\; i = 0, 1, \ldots, n.
\end{equation}
In \eqref{eq:3.1}, the  integral term can be estimated by a numerical method of integration, e.g. Newton--Cotes methods. Therefore, by taking equidistant partition $\mathsf{\Pi}$, as above with $h = s_{i+1} - s_i, \;i = 0, 1, \ldots, n-1$ and also the known weights $w_{i_j}, \; j = 0, 1, \ldots, i$ for interval $[a,t_i]$ and $w_r, \; r = 0, 1, \ldots, n$ for interval $[a,b]$, equality \eqref{eq:3.1} can be written as
\begin{align}
 x^*_i  +  & \sum\limits_{j = 0}^{i }{w_{i_j} K_1(t_i, s_j, x^*_j)} + O(h^{\nu_1})  +  \sum\limits_{r = 0}^{n }{w_r K_2(t_i, s_r, x^*_r)} + O(h^{\nu_1})= g_i, \nonumber \\
& i = 0, 1, \ldots, n, \label{eq:3.2}
\end{align}
 where $x^*_i = x^*(t_i), g_i = g(t_i),\; i = 0, 1, \ldots, n$,  $2 \le \nu_1 \le \nu $   and  depend upon the used method of Newton--Cotes for estimating  the integrals in \eqref{eq:3.1}. 
  From \eqref{eq:3.2},  we have
\begin{equation}\label{eq:3.3}
x^*_i +    \sum\limits_{j = 0}^{i }{w_{i_j} K_1(t_i, s_j, x^*_j)} +     \sum\limits_{r = 0}^{n }{w_r K_2(t_i, s_r, x^*_r)} + O(h^{\nu_1})  = g_i, \;i = 0, 1,  \ldots, n.
\end{equation}
\indent  For partition $\mathsf{\Pi}$, we consider a perturbed  system of nonlinear equations  obtained by neglecting the truncation error of \eqref{eq:3.1} as follows
\begin{equation}\label{eq:3.4}
\xi_i +    \sum\limits_{j = 0}^{i }{w_{i_j} K_1(t_i, s_j, \xi_j)}  +  \sum\limits_{r = 0}^{n }{w_r K_2(t_i, s_r, \xi_r)}  = g_i, \; i = 0, 1,  \ldots, n.
\end{equation}
The perturbed system of nonlinear equations \eqref{eq:3.4} can be rewritten as 
\begin{equation}\label{eq:3.5}
\xi  + \Phi(\xi) + F(\xi) = g,
\end{equation}
where $ \xi =(\xi_0, \xi_1, \ldots,\xi_n)^T,  g = (g_0, g_1, \ldots, g_n)^T,  \Phi (\xi) = (\varphi_0(\xi), \varphi_1(\xi), \ldots,$ $ \varphi_n(\xi))^T $ and
$F(\xi) = (f_0(\xi), f_1(\xi), \ldots, f_n(\xi))^T $
with 
\[ \varphi_i(\xi) =  \sum\limits_{j = 0}^{i}{w_{i_j} K_1(t_i, s_j, \xi_j)}, \; f_i(\xi) =  \sum\limits_{r = 0}^{n }{w_r K_2(t_i, s_r, \xi_r)},\; i= 0, 1, \ldots, n.\]  
\indent  For partition $\mathsf{\Pi}$ and the known weights $w_i, \;i = 0, 1, \ldots, n$, we define an inner product in  $\mathbb{R}^{n+1}$ by 
\[ \left\langle \xi, \overline{\xi}\right\rangle = \sum\limits_{i = 0}^n{w_i \, \xi_i\,  \overline{\xi}_i}, \; \forall  \xi =(\xi_0, \xi_1, \ldots,\xi_n)^T, \overline{\xi} = (\overline{\xi}_0, \overline{\xi}_1, \ldots, \overline{\xi}_n)^T \in \mathbb{R}^{n+1}.\]
This inner product induces the norm
\[ \| \xi\| = \sqrt{\left\langle\xi,\xi \right\rangle } =  \left( \sum\limits_{i = 0}^n{w_i |\xi_i|^2} \right)^{\frac{1}{2}}. \]
\indent The following proposition gives us the property of the mapping $\Phi$.
\begin{proposition}\label{pro2.1}
	Let the assumption {\rm (iii)} be satisfied. Then 
	\begin{equation}\label{eq:3.6}
	\| \Phi^m (\xi)  - \Phi^m (\overline{\xi}) \| \le \frac{M^m}{ \sqrt{(m-1)!} } \|\xi - \overline{\xi} \|,\; \forall \xi, \overline{\xi} \in \mathbb{R}^{n+1}
	\end{equation}
	for some positive integer $m$.
\end{proposition}
\begin{proof} 
	Define the operator $V$ as 
	\[  (Vx)(t)= \int\limits_a^t {K_1(t,s,x(s)) ds}, \; \forall x(t) \in L^2[a,b].    \]
	First, we prove that \eqref{eq:3.6} is hold for $m=1$. From (iii),  we have
	\begin{gather*}
	\begin{split}
	 |(Vx)(t) - (V\overline{x})(t)|  & = \left| \int\limits_a^t {\left[ K_1(t,s,x(s))  - K_1(t,s,\overline{x}(s)) \right] ds} \right| \\
	& \le  \int\limits_a^t{ \left| K_1(t,s,x(s))  - K_1(t,s,\overline{x}(s)) \right| ds} \\
	& \le   \int\limits_a^t{ \left|\psi(t,s) \right|  \left| x(s)  - \overline{x}(s) \right| ds}
	\end{split} 
	\end{gather*}
	for all $x(t), \overline{x}(t) \in C^{\nu}[a,b]$. 
	From this and Cauchy--Schwarz inequality, we obtain
	\vspace*{-0.3truecm}
	\begin{gather}\label{eq:3.7}
	\begin{split}
	 |(Vx)(t) - (V\overline{x})(t)|^2  &  \le  \int\limits_a^t{\left|\psi(t,s) \right|^2 ds }  \int\limits_a^t{  \left| x(s)  - \overline{x}(s) \right|^2 ds}\\
	&  \le Q^2(t) \int\limits_a^t{  \left| x(s)  - \overline{x}(s) \right|^2 ds} \\
	& = Q^2(t) \int\limits_a^t{  \left| x(s_1)  - \overline{x}(s_1) \right|^2 ds_1}
	\end{split} 
	\end{gather}
and hence
	\vspace*{-0.3truecm}
	\[ |(Vx)(t) - (V\overline{x})(t)|^2 \le   Q^2(t) \int\limits_a^b{  \left| x(s_1)  - \overline{x}(s_1) \right|^2 ds_1}.     \]
Then we have
	\vspace*{-0.1truecm}
	\begin{gather*}
	\begin{split}
	 \int\limits_a^b{\left| (Vx)(t) - (V\overline{x})(t) \right|^2 dt}  
	& \le  \int\limits_a^b{Q^2(t)dt} \int\limits_a^b{  \left| x(t)  - \overline{x}(t) \right|^2 dt} \\
	& \le M^2 \int\limits_a^b{  \left| x(t)  - \overline{x}(t) \right|^2 dt}.
	\end{split} 
	\end{gather*}
This implies that
	\[ \int\limits_a^b{\left| \int\limits_a^t {[ K_1(t,s,x(s))  - K_1(t,s,\overline{x}(s)) ] ds} \right|^2 dt } - M^2 \int\limits_a^b{  | x(t)  - \overline{x}(t) |^2 dt} \le  0\]
	for all $x(t), \overline{x}(t) \in C^{\nu}[a,b]$. We may assume without loss of generality that
	\begin{equation}\label{eq:3.8}
	\int\limits_a^b{\left| \int\limits_a^t {[ K_1(t,s,x(s))  - K_1(t,s,\overline{x}(s))] ds} \right|^2 dt } - M^2 \int\limits_a^b{  | x(t)  - \overline{x}(t) |^2 ds}< 0
	\end{equation} 
	for all $x(t), \overline{x}(t) \in C^{\nu}[a,b]$ with $x(t)\ne  \overline{x}(t)$.\\
	\indent By taking equidistant partition $\mathsf{\Pi}$, as above with $h = t_{i+1} - t_i, \; i = 0, 1, \ldots, n-1$ and  also the known weights $w_i , w_e, \; i, e = 0, 1, \ldots, n$ for interval $[a,b]$,  we have
	\begin{gather}\label{eq:3.9}
	\begin{split}
	& \sum\limits_{i = 0}^{n }{w_i \left| \int\limits_a^{t_i} {\left[ K_1(t_i,s,x(s))  - K_1(t_i,s,\overline{x}(s)) \right] ds} \right|^2 } + O(h^{\nu_1}) \\
	& - M^2 \sum\limits_{e = 0}^{n }{w_e |x_e - \overline{x}_e|^2 } 
	-  O(h^{\nu_1}) < 0,  
	\end{split}    
	\end{gather}
	where $x_e = x(t_e), \overline{x}_e = \overline{x}(t_e), \; e = 0, 1, \ldots, n$ and  $2 \le \nu_1 \le  \nu$ depend upon the used method of Newton--Cotes for estimating the integrals in \eqref{eq:3.8}. From \eqref{eq:3.9}, we have
	\begin{gather}\label{eq:3.10}
      \begin{split} 
	& \sum\limits_{i = 0}^{n }{w_i \left| \int\limits_a^{t_i} {\left[ K_1(t_i,s,x(s))  - K_1(t_i,s,\overline{x}(s)) \right] ds} \right|^2 }   \\
	& - M^2 \sum\limits_{e = 0}^{n }{w_e |x_e - \overline{x}_e|^2 } 
	+  O(h^{\nu_1}) < 0.
		\end{split}    
	\end{gather}
	 Therefore, by taking equidistant partition $\mathsf{\Pi}$, as above with $h = s_{i+1} - s_i,\; i = 0, 1, \ldots, n-1$ and also the known weights $w_{i_j}, \; j = 0, 1, \ldots, i$ for interval $[a,t_i]$, the inequality \eqref{eq:3.10} can be written as
	\begin{align*} 
	& \sum\limits_{i = 0}^{n}{w_i  \left|{ \sum\limits_{j = 0}^{i }{w_{i_j} [ K_1(t_i,s_j,x_j) - K_1(t_i,s_j, \overline{x}_j) ]} + O(h^{\nu_1})  } \right|^2 } \\
	& - M^2 \sum\limits_{e = 0}^{n }{w_e |x_e - \overline{x}_e|^2 } 
	+  O(h^{\nu_1}) < 0, 
	\end{align*}
	where $x_j = x(s_j),  \overline{x}_j = \overline{x}(s_j), j = 0, 1, \ldots, i$ and $2 \le \nu_1 \le \nu$  depend upon the used method of Newton--Cotes for estimating  the integral in \eqref{eq:3.10}. Hence, for sufficiently large $n$, we have
	\begin{equation*} 
	\sum\limits_{i = 0}^{n}{w_i  \left|{ \sum\limits_{j = 0}^{i }{w_{i_j} [ K_1(t_i,s_j,\xi_j) - K_1(t_i,s_j, \overline{\xi}_j) ]}  } \right|^2 }  - M^2 \sum\limits_{e = 0}^{n }{w_e |\xi_e - \overline{\xi}_e|^2 } 
	\le 0
	\end{equation*}
	for all $\xi, \overline{\xi} \in \mathbb{R}^{n+1}$.	That means
	\[\| \Phi (\xi)  - \Phi (\overline{\xi}) \|^2 - M^2 \|\xi - \overline{\xi} \|^2 \le 0, \;  \forall \xi, \overline{\xi} \in \mathbb{R}^{n+1}.\]
Hence
	\[\| \Phi (\xi)  - \Phi (\overline{\xi}) \| \le M \|\xi - \overline{\xi} \|, \; \forall \xi, \overline{\xi} \in \mathbb{R}^{n+1}. \]
	Consequently,  \eqref{eq:3.6} is hold for $m=1$. \\
	\indent  We now prove that \eqref{eq:3.6} is hold for $m=2$.	From \eqref{eq:3.7}, we have
	\begin{gather}\label{eq:3.12} 
	\begin{split} 
	|(V^2x)(t) - (V^2\overline{x})(t)|^2 
	& \le Q^2(t)  \int\limits_a^t {|  (Vx)(s_1)  - (V\overline{x})(s_1) |^2 ds_1}   \\
	& \le  Q^2(t) \int\limits_a^t{Q^2(s_1)ds_1} \int\limits_a^{s_1}{  | x(s_2)  - \overline{x}(s_2) |^2 ds_2}
		\end{split}    
	\end{gather}
	and hence 
	\[|(V^2x)(t) - (V^2\overline{x})(t)|^2    \le Q^2(t) \int\limits_a^t{Q^2(s_1)ds_1} \int\limits_a^b{  | x(s_2)  - \overline{x}(s_2) |^2 ds_2}. \]
	Then we have
	\begin{align*}
	 \int\limits_a^b{|(V^2x)(t) - (V^2\overline{x})(t)|^2 dt}  &  \le \int\limits_a^b{Q^2(t) dt}   \int\limits_a^b{Q^2(s_1)ds_1} \int\limits_a^b{  | x(t)  - \overline{x}(t) |^2 dt} \\
	& \le M^4 \int\limits_a^b{| x(t)  - \overline{x}(t) |^2 dt}.
	\end{align*} 
	It follows that
	\begin{align*}
	& \int\limits_a^b{\left| \int\limits_a^t {[ K_1(t,s,(Vx)(s))  - K_1(t,s,(V\overline{x})(s)) ] ds} \right|^2 dt }\\
	&  - M^4 \int\limits_a^b{  | x(t)  - \overline{x}(t)|^2 dt} \le 0
	\end{align*} 
	for all $x(t), \overline{x}(t) \in C^{\nu}[a,b]$.
	We may assume without loss of generality that 
	\begin{gather}\label{eq:3.13}
	 \begin{split} 
	& \int\limits_a^b{\left| \int\limits_a^t {[ K_1(t,s,(Vx)(s))  - K_1(t,s,(V\overline{x})(s)) ] ds} \right|^2 dt } \\
	& - M^4 \int\limits_a^b{  | x(t)  - \overline{x}(t) |^2 dt} < 0
	 \end{split} 
	\end{gather}
	for all $x(t), \overline{x}(t) \in C^{\nu}[a,b]$ with  $x(t) \ne \overline{x}(t)$.\\
	\indent By taking equidistant partition $\mathsf{\Pi}$, as above with $h = t_{i+1} - t_i,\; i = 0, 1, \ldots, n-1$ and  also the known weights $w_i , w_e, \; i, e = 0, 1, \ldots, n$ for interval $[a,b]$,  we have
	\begin{gather}\label{eq:3.14}
	\begin{split}
	& \sum\limits_{i = 0}^{n }{w_i \left| \int\limits_a^{t_i} {\left[ K_1(t_i,s,(Vx)(s))  - K_1(t_i,s,(V\overline{x})(s)) \right] ds} \right|^2 } + O(h^{\nu_1}) \\
	& - M^4 \sum\limits_{e = 0}^{n }{w_e |x_e - \overline{x}_e|^2 } 
	-  O(h^{\nu_1}) < 0,  
	\end{split}    
	\end{gather}
	where $x_e = x(t_e), \overline{x}_e = \overline{x}(t_e), \; e = 0, 1, \ldots, n$ and  $2 \le \nu_1 \le \nu$ depend upon the used method of Newton--Cotes for estimating  the integrals in \eqref{eq:3.13}.
	From \eqref{eq:3.14}, we have
	\begin{gather}\label{eq:3.15}
	\begin{split}
	& \sum\limits_{i = 0}^{n }{w_i \left| \int\limits_a^{t_i} {\left[ K_1(t_i,s,(Vx)(s))  - K_1(t_i,s,(V\overline{x})(s)) \right] ds} \right|^2 }   \\
	& - M^4 \sum\limits_{e = 0}^{n }{w_e |x_e - \overline{x}_e|^2 }  +  O(h^{\nu_1}) < 0.  
	\end{split}    
	\end{gather}
	Therefore, by taking equidistant partition $\mathsf{\Pi}$, as above with $h = s_{i+1} - s_i, \;  i = 0, 1, \ldots, n-1$ and also the known weights $w_{i_j}, \; j = 0, 1, \ldots, i$ for interval $[a,t_i]$, the inequality \eqref{eq:3.15} can be written as
	\begin{gather} \label{eq:3.16}
	\begin{split}
	  &  \sum\limits_{i = 0}^{n}{w_i \left|{ \sum\limits_{j = 0}^{i }{ w_{i_j} [ K_1 (t_i,s_j,\int\limits_a^{s_j}{K_1(s_j, \varsigma, x(\varsigma))  d\varsigma} ) } } \right.}  \\
	& \;\;\;\;\;\;\;\;\;\;\;\;\;\;\;\;\;\;\;\;\;\;\;\; \left.{ - K_1 (t_i,s_j, \int\limits_a^{s_j}{K_1(s_j, \varsigma, \overline{x}(\varsigma) ) d\varsigma} ) ] +  O(h^{\nu_1})  } \right|^2  \\
	& - M^2 \sum\limits_{e = 0}^{n }{w_e |x_e - \overline{x}_e|^2 } 
	+  O(h^{\nu_1}) < 0, 
	\end{split} 
	\end{gather}
where	$2 \le \nu_1 \le \nu$ depend upon the used method of Newton--Cotes for estimating the integrals in \eqref{eq:3.15}. 
 By taking equidistant partition $\mathsf{\Pi}$, as above with $h = \varsigma_{i+1} - \varsigma_i, \; i = 0, 1, \ldots, n-1$ and also the known weights $ w_{j_{\rho}}, w_{j_{\rho'}}, \; \rho, \rho' = 0, 1, \ldots, j$ for interval $[a,s_j]$, the inequality \eqref{eq:3.16} can be written as
 	\begin{align*} 
	& \sum\limits_{i = 0}^n{ w_i  \left|{ \sum\limits_{j = 0}^i{ w_{i_j} [ K_1 (t_i,s_j,\sum\limits_{\rho = 0}^j {w_{j_{\rho}}K_1(s_j, \varsigma_{\rho}, x_{\rho}) } +  O(h^{\nu_1}) )  }   }\right. }    \\
	& \;\;\;\;\;\;\;\;\;\;\;\;\;\;\;\;\;\;\;\;\;\;\;\; \left.{  - K_1 (t_i,s_j, \sum\limits_{\rho' = 0}^{j }{w_{j_{\rho'}}K_1(s_j, \varsigma_{\rho'}, \overline{x}_{\rho'}) +  O(h^{\nu_1}) } ) ] + O(h^{\nu_1})  } \right|^2  \\
	&  - M^2 \sum\limits_{e = 0}^{n }{w_e |x_e - \overline{x}_e|^2 } +  O(h^{\nu_1}) < 0. 
	\end{align*}
	where $x_{\rho} = x(\varsigma_{\rho}), \overline{x}_{\rho'} = \overline{x}(\varsigma_{\rho'}),\, \rho, \rho' = 0, 1, \ldots, j$ and $2 \le \nu_1 \le \nu$ depend upon the used method of Newton--Cotes for estimating the integrals in \eqref{eq:3.16}. Therefore, for sufficiently large $n$,  we have
	\begin{align*} 
	& \sum\limits_{i = 0}^n{ w_i  \left|{ \sum\limits_{j = 0}^i{ w_{i_j} [ K_1 (t_i,s_j,\sum\limits_{\rho = 0}^j {w_{j_{\rho}}K_1(s_j, \varsigma_{\rho}, \xi_{\rho}) }  )}  }\right. }    \\
	&\;\;\;\;\;\;\;\;\;\;\; \left.{  - K_1 (t_i,s_j, \sum\limits_{\rho' = 0}^{j }{w_{j_{\rho'}}K_1(s_j, \varsigma_{\rho'}, \overline{\xi}_{\rho'})    } ) ]    }\right|^2  
   - M^2 \sum\limits_{e = 0}^{n }{w_e |\xi_e - \overline{\xi}_e|^2 }  \le 0
	\end{align*}
	for all $\xi, \overline{\xi} \in \mathbb{R}^{n+1}$. That means
	\[\| \Phi^2 (\xi)  - \Phi^2 (\overline{\xi}) \|^2 - M^4 \|\xi - \overline{\xi} \|^2 \le  0, \; \forall \xi, \overline{\xi} \in \mathbb{R}^{n+1}.\] 
Thus
	\[\| \Phi^2 (\xi)  - \Phi^2 (\overline{\xi}) \| \le M^2 \|\xi - \overline{\xi} \|, \; \forall \xi, \overline{\xi} \in \mathbb{R}^{n+1}. \]
Consequently,  \eqref{eq:3.6} is hold for $m=2$. \\
	\indent Next, we shall prove that \eqref{eq:3.6} is hold for some integer $m \ge 3$. Continuing the process from  \eqref{eq:3.7} and   \eqref{eq:3.12}, we obtain
		\vspace*{-0.2truecm}
	\begin{gather*}
	\begin{split}
	& \!{|(V^{m}x)(t) - (V^{m}\overline{x})(t) |^2 \le Q^2(t)  \int\limits_a^t {|  (V^{m-1} x)(s_1)  - (V^{m-1} \overline{x})(s_1) |^2 ds_1}  \le Q^2(t) }\\
	& \!{ \times  \int\limits_a^t{Q^2(s_1)ds_1} \int\limits_a^{s_1}{Q^2(s_2)ds_2}
	\cdots  \int\limits_a^{s_{m-2}}{Q^2(s_{m-1})ds_{m-1}} \int\limits_a^{s_{m-1}}{ |x(s_{m})  - \overline{x}(s_{m}) |^2 ds_{m}} }
	\end{split} 
	\end{gather*}
	and hence
		\vspace*{-0.1truecm}
	\begin{gather}\label{eq:3.17}
	\begin{split}
	& |(V^{m}x)(t) - (V^{m}\overline{x})(t) |^2 \le Q^2(t) \\
	& \times \!{\int\limits_a^t{Q^2(s_1)ds_1} \int\limits_a^{s_1}{Q^2(s_2)ds_2}  
	\cdots \int\limits_a^{s_{m-2}}{Q^2(s_{m-1}) ds_{m-1}} \int\limits_a^b{| x(s_m)  - \overline{x}(s_m) |^2 ds_m} }.
	\end{split} 
	\end{gather}
 By induction, we can show that
 	\vspace*{-0.3truecm}
\begin{gather}\label{eq:3.18}
	\begin{split}
	& \int\limits_a^t{Q^2(s_1)ds_1} \int\limits_a^{s_1}{Q^2(s_2)ds_2} \cdots \int\limits_a^{s_{m-2}}{Q^2(s_{m-1}) ds_{m-1}} \\
	& = \frac{1}{(m-1)!}  \left( \int\limits_a^t{Q^2(s)ds} \right)^{m-1}.
		\end{split} 
	\end{gather}
	Combining now \eqref{eq:3.17} and \eqref{eq:3.18}, we get
	\begin{gather*}
	\begin{split}
	& | (V^{m}x)(t) - (V^{m}\overline{x})(t) |^2    \\
	& \le Q^2(t)\frac{1}{(m-1)!} \left( \int\limits_a^t{Q^2(s)ds}\right)^{m-1} \int\limits_a^b{  | x(s_m)  - \overline{x}(s_m) |^2 ds_m} \\
	& \le Q^2(t) \frac{1}{(m-1)!} \left( \int\limits_a^b{Q^2(s)ds}\right)^{m-1}  \int\limits_a^b{| x(s_m)  - \overline{x}(s_m)|^2 ds_m}\\
	& 	\le Q^2(t)  \frac{M^{2(m-1)}}{(m-1)!} \int\limits_a^b{ | x(s_m)  - \overline{x}(s_m) |^2 ds_m}.
	\end{split} 
	\end{gather*}
	Hence
\[	 \int\limits_a^b{\left| (V^{m}x)(t) - (V^{m}\overline{x})(t) \right|^2 dt}  
	 \le \frac{M^{2(m-1)}}{(m-1)!}  \int\limits_a^b{ Q^2(t)dt  }  \int\limits_a^b{  \left| x(t)  - \overline{x}(t) \right|^2 dt}  \]
	 \newpage
	\[ \le \frac{M^{2m}}{(m-1)!} \int\limits_a^b{  \left| x(t)  - \overline{x}(t) \right|^2 dt}\]
		for all $x(t), \overline{x}(t) \in C^{\nu}[a,b]$.
	In a similar way as above, we can show that
	\[\| \Phi^m (\xi)  - \Phi^m (\overline{\xi}) \|^2  - \frac{M^{2m}}{(m-1)!}  \|\xi - \overline{\xi} \|^2 \le 0, \; 
	\forall \xi, \overline{\xi} \in \mathbb{R}^{n+1}. \]
	This implies that
	\[\| \Phi^m (\xi)  - \Phi^m (\overline{\xi}) \| \le \frac{M^m}{ \sqrt{(m-1)!} } \|\xi - \overline{\xi} \|, \; \forall \xi, \overline{\xi} \in \mathbb{R}^{n+1}.  \]
	Consequently, \eqref{eq:3.6} is hold for some  integer $m \ge 3$. This completes the proof of the proposition.
\end{proof} 
\indent We now give some properties of the mapping $F$ in the following proposition.
\begin{proposition}\label{pro2.2}
	Let the assumptions {\rm (iv)} and {\rm (v)}  be satisfied. Then 
	\begin{equation}\label{eq:3.19}
	\| F (\xi)  - F (\overline{\xi}) \| \le L \|\xi - \overline{\xi} \|, \; \forall \xi, \overline{\xi} \in \mathbb{R}^{n+1}
	\end{equation}	
	and
	\begin{equation}\label{eq:3.20}
	\left\langle F(\xi) - F(\overline{\xi}), \xi -\overline{\xi} \right\rangle >0,  \; \forall \xi, \overline{\xi} \in \mathbb{R}^{n+1}, \; \xi \ne \overline{\xi}.
	\end{equation}
\end{proposition}
\begin{proof}
	From (iv),  we have
	\begin{gather*}
	\begin{split}
	\left| \int\limits_a^b {[ K_2(t,s,x(s))  - K_2(t,s,\overline{x}(s)) ] ds} \right|  & \le  \int\limits_a^b{ | K_2(t,s,x(s))  - K_2(t,s,\overline{x}(s)) | ds}\\
	& \le   \int\limits_a^b{ |\phi(t,s) |  | x(s)  - \overline{x}(s) | ds}
	\end{split} 
	\end{gather*}
	for all $x(t), \overline{x}(t) \in C^{\nu}[a,b]$. 
	From this and Cauchy--Schwarz inequality, we obtain
	\begin{gather*}
	\begin{split}
	 &\int\limits_a^b{\left| \int\limits_a^b {[ K_2(t,s,x(s))  - K_2(t,s,\overline{x}(s)) ] ds} \right|^2dt}\\
	 &  \le  \int\limits_a^b \int\limits_a^b {|\phi(t,s) |^2dsdt} \int\limits_a^b{| x(s)  - \overline{x}(s) |^2 ds}  =  L^2 \int\limits_a^b{| x(t)  - \overline{x}(t) |^2 dt}
	\end{split} 
	\end{gather*}
and hence
	\[\int\limits_a^b{\left| \int\limits_a^b {\left[ K_2(t,s,x(s))  - K_2(t,s,\overline{x}(s)) \right] ds} \right|^2dt}  - L^2 \int\limits_a^b{\left| x(t)  - \overline{x}(t) \right|^2 dt} \le 0. \]
	We may assume without loss of generality that
	\begin{equation}\label{eq:3.21}
	\int\limits_a^b{\left| \int\limits_a^b {\left[ K_2(t,s,x(s))  - K_2(t,s,\overline{x}(s)) \right] ds} \right|^2dt}  -  L^2 \int\limits_a^b{\left| x(t)  - \overline{x}(t) \right|^2 dt} < 0
	\end{equation}
	for all $x(t), \overline{x}(t) \in C^{\nu}[a,b]$ with $x(t) \ne \overline{x}(t)$.\\
\indent By taking equidistant partition $\mathsf{\Pi}$, as above with $h = t_{i+1} - t_i,\; i = 0, 1, \ldots, n-1$ and also the known weights $w_i, w_e, \; i, e = 0, 1, \ldots, n$ for interval $[a,b]$,  we have
	\begin{gather}\label{eq:3.22}
		\begin{split}
	& \sum\limits_{i = 0}^{n }{w_i \left| \int\limits_a^b {\left[ K_2(t_i,s,x(s))  - K_2(t_i,s,\overline{x}(s)) \right] ds} \right|^2 } + O(h^{\nu_1}) \\
	& - L^2 \sum\limits_{e = 0}^{n }{w_e |x_e - \overline{x}_e|^2 }   -  O(h^{\nu_1}) < 0,  
	\end{split} 
	\end{gather}
	where $x_e = x(t_e), \overline{x}_e = \overline{x}(t_e),\; e = 0, 1, \ldots, n$ and  $2 \le \nu_1 \le \nu$ depend upon the used method of Newton--Cotes for estimating the integrals in \eqref{eq:3.21}.   From \eqref{eq:3.22}, we have
	\begin{gather}\label{eq:3.23}
   \begin{split}
	& \sum\limits_{i = 0}^{n }{w_i \left| \int\limits_a^b {\left[ K_2(t_i,s,x(s))  - K_2(t_i,s,\overline{x}(s)) \right] ds} \right|^2 } \\
	& - L^2 \sum\limits_{e = 0}^{n }{w_e |x_e - \overline{x}_e|^2 } + O(h^{\nu_1}) < 0.
	\end{split} 
  \end{gather}
 Therefore, by taking equidistant partition $\mathsf{\Pi}$, as above with $h = s_{i+1} - s_i,\; i = 0, 1, \ldots, n-1$ and also the known weights $w_r, \; r = 0, 1, \ldots, n$ for interval $[a,b]$, the inequality \eqref{eq:3.23} can be written as
	\begin{gather*}
\begin{split}
	& \sum\limits_{i = 0}^{n}{w_i  \left|\sum\limits_{r = 0}^{n }{w_r \left[ K_2(t_i,s_r,x_r) - K_2(t_i,s_r,\overline{x}_r)\right]} + O(h^{\nu_1})   \right|^2 }  \\
	& - L^2 \sum\limits_{e = 0}^{n }{w_e |x_e - \overline{x}_e|^2 } 
	 + O(h^{\nu_1}) < 0, 
	\end{split} 
\end{gather*}
	where $x_r = x(s_r), \overline{x}_r = \overline{x}(s_r),\; r = 0, 1, \ldots, n$ and $2 \le \nu_1 \le \nu$ depend upon the used method of Newton--Cotes for estimating  the integral in \eqref{eq:3.23}. Hence, for sufficiently large $n$,  we have
	\[\sum\limits_{i = 0}^{n}{w_i  \left|\sum\limits_{r = 0}^{n }{w_r \left[ K_2(t_i,s_r,\xi_r) - K_2(t_i,s_r,\overline{\xi}_r)\right]} \right|^2 }\\
	- L^2 \sum\limits_{e = 0}^{n }{w_e |\xi_e - \overline{\xi}_e|^2 } \le 0   \]
	 for all $\xi, \overline{\xi} \in \mathbb{R}^{n+1}$. That means
	\[\| F (\xi)  - F (\overline{\xi}) \|^2  - L^2  \|\xi - \overline{\xi} \|^2 \le 0, \; \forall \xi, \overline{\xi} \in \mathbb{R}^{n+1}.\]
	 Hence
	\[\| F (\xi)  - F (\overline{\xi}) \| \le L \|\xi - \overline{\xi} \|, \forall \xi, \overline{\xi} \in \mathbb{R}^{n+1}, \]
	which proves \eqref{eq:3.19}.\\
	\indent Let us prove \eqref{eq:3.20}. From  (v), by taking equidistant partition $\mathsf{\Pi}$, as above with $h = t_{i+1} - t_i,\; i = 0, 1, \ldots, n-1$,  and
	also the known weights $w_i,\; i = 0, 1, \ldots, n$ for interval $[a,b]$, we obtain
	\begin{equation}\label{eq:3.24}
	\sum\limits_{i = 0}^{n }{w_i \int\limits_a^b { \left[ K_2(t_i,s,x(s)) - K_2(t_i,s,\overline{x}(s))\right] ds } \;  [x_i - \overline{x}_i ] } + O(h^{\nu_1}) > 0,
	\end{equation}
	where $x_i = x(t_i), \overline{x}_i = \overline{x}(t_i),\;  i = 0, 1, \ldots, n$ and $2 \le \nu_1 \le \nu$ depend upon the used method of Newton--Cotes for estimating the integral.   
	Therefore, by taking equidistant partition $\mathsf{\Pi}$, as above with $h = s_{i+1} - s_i, \; i = 0, 1, \ldots, n-1$ and also the known weights $w_r, \; r = 0, 1, \ldots, n$ for interval $[a,b]$, the inequality \eqref{eq:3.24} can be written as
	\begin{equation*} 
	 \sum\limits_{i = 0}^{n }{w_i \left\{ \sum\limits_{r = 0}^{n }{ w_r [ K_2(t_i,s_r,x_r) \!-\! K_2(t_i,s_r,\overline{x}_r)]} \!+\! O(h^{\nu_1}) \right\}[x_i \!-\! \overline{x}_i ] } 
	 + O(h^{\nu_1}) > 0, 
	\end{equation*}
	where $x_r = x(s_r), \overline{x}_r = \overline{x}(s_r), \; r = 0, 1, \ldots, n$ and $2 \le \nu_1 \le \nu$ depend upon the used method of Newton--Cotes for estimating the integral in \eqref{eq:3.24}. Hence, for sufficiently large $n$, we have
	\[  \sum\limits_{i = 0}^{n }{w_i  \left\{  \sum\limits_{r = 0}^{n }{ w_r \left[ K_2(t_i,s_r,\xi_r) - K_2(t_i,s_r,\overline{\xi}_r)\right]} \right\}   [\xi_i - \overline{\xi}_i ] } >0 \]
	for all $\xi, \overline{\xi} \in \mathbb{R}^{n+1}$ with $\xi \ne \overline{\xi}$. That means
	\[\left\langle F(\xi) - F(\overline{\xi}), \xi -\overline{\xi} \right\rangle > 0,\;  \forall \xi, \overline{\xi} \in \mathbb{R}^{n+1}, \; \xi \ne \overline{\xi}.\]
	 This completes the proof of the proposition. 
\end{proof} 
In order to prove our main results, we need the following theorems.
\begin{theorem}\label{thm3.1}
	Assume $H$ is a nonempty closed set in a Banach space $X$ and $T: H \to H$ is continuous. Suppose that $T^m$ is a contractive operator for some positive integer $m$. Then $T$ has a unique fixed point $x^*$ in $H$. Moreover, the iteration process
	\begin{equation}\label{eq:3.25}
	x_{k+1} = T (x_k), \;\;\; k = 0, 1, 2, \ldots
	\end{equation}
	converges to the fixed point $x^*$.
\end{theorem}
\begin{proof}
	For proof see \cite{RefBAtki} or \cite{RefJMalae}.	 
\end{proof}
\begin{theorem}\label{thm3.2}
	Let the assumptions of Theorem \ref{thm3.1} be satisfied and let $\{x_k\},\; k = 1, 2, \ldots$ be constructed by iteration process \eqref{eq:3.25}. Then for $k \ge m$, the following estimates hold
	\begin{equation}\label{eq:3.26}
	\| x_k-x^* \| \le \frac{\alpha^{\frac{k-h_0}{m}}}{1-\alpha} \| x_{m+h_0}-x_{h_0} \|,
	\end{equation}
	where  $\alpha$ is the contraction coefficient of the operator $T^m,  h_0 \in \left\lbrace 0, 1, \ldots, m-1 \right\rbrace$ is the residual of $\frac{k}{m}$.
\end{theorem}
\begin{proof}
	For proof see  \cite{RefJMalae}.	 
\end{proof} 
\indent Now, we shall give the existence and uniqueness of the solution of the perturbed system of nonlinear equations \eqref{eq:3.5}.
\begin{theorem}\label{thm3.3}
	Let the assumptions {\rm (i)--(v)}  be satisfied. Then the perturbed system of nonlinear equations \eqref{eq:3.5} has a unique solution for any $g \in \mathbb{R}^{n+1}$.
\end{theorem}
\begin{proof} 
	We shall carry out a change of variable
	\begin{equation}\label{eq:3.27}
	z = \xi + F(\xi) \equiv P(\xi).
	\end{equation}
	By  Proposition \ref{pro2.2},  the mapping $F$ is monotone and Lipschitz - continuous with  Lipschitz coefficient equal to $L$. Therefore, by  Theorem \ref{thm2.1}, the  system of equations \eqref{eq:3.27} has a unique solution for any $z \in \mathbb{R}^{n+1}$,\; i.e., the mapping $P^{-1}(z)$ is  determined in the whole space $\mathbb{R}^{n+1}$.  By virtue of the monotonicity of the mapping $F$,  the mapping $P^{-1}$ is Lipschitz - continuous with Lipschitz coefficient equal to $1$. Indeed, for all  $z, \overline{z} \in  \mathbb{R}^{n+1}$ we have
	\[ \| P^{-1}(z) - P^{-1}(\overline{z})  \|  = \|\xi - \overline{\xi}  \|  \le  \|\xi - \overline{\xi} +  F(\xi) - F(\overline{\xi}) ]  \|  = \|z - \overline{z}  \|.	 \]
	After changing the variable \eqref{eq:3.27}, the perturbed system of nonlinear equations \eqref{eq:3.5} will take the following form
	\begin{equation}\label{eq:3.28}
	z + \Phi P^{-1}(z)  = g. 
	\end{equation}
	Define the mapping $T$ as
	\begin{equation}\label{eq:3.29}
	T(z) = - \Phi P^{-1}(z) + g, \; \forall z \in \mathbb{R}^{n+1}.
	\end{equation}
	Then the system of equations \eqref{eq:3.28} can be rewritten as
	\begin{equation}\label{eq:3.30}
	z  = T(z).
	\end{equation}
	It follows from  \eqref{eq:3.29} that for all $z, \overline{z} \in  \mathbb{R}^{n+1}$	 and for some positive integer $m$,
	\[ \|T^m(z) - T^m(\overline{z}) \| = \| (- \Phi P^{-1})^m (z) - (- \Phi P^{-1})^m(\overline{z})  \|.     \]
	By virtue of Proposition \ref{pro2.1} and Lipschitz continuity of the mapping $P^{-1}$, we have
	\[\|T^m(z) - T^m(\overline{z}) \| \le \frac{M^m}{ \sqrt{(m-1)!} } \|z - \overline{z} \|, \; \forall z, \overline{z} \in \mathbb{R}^{n+1}. \]
	Since $\frac{M^m}{ \sqrt{(m-1)!} } < 1 $ when $m$ is sufficiently large, we see that  $T^m$ is a contractive mapping with contraction coefficient equal to $\alpha = \frac{M^m}{ \sqrt{(m-1)!} }$. By Theorem \ref{thm3.1}, the mapping $T$ has a unique fixed point $z^* \in \mathbb{R}^{n+1}$,  i.e., the system of equations \eqref{eq:3.28} has a unique solution $z^* \in \mathbb{R}^{n+1}$ for any $g \in \mathbb{R}^{n+1}$. Consequently, the perturbed system of nonlinear equations \eqref{eq:3.5} has a unique solution $\xi^* $ for any $g \in \mathbb{R}^{n+1}$. This completes the proof of the theorem. 
\end{proof}
\indent In the following proposition, we shall estimate $ \|x^* - \xi^*\|$, where $x^*=(x^*_0, x^*_1, \ldots,x^*_n)^T$ with $x^*_i = x^*(t_i),\; i = 0, 1, \ldots, n$ (note that $x^*(t)$ is an analytical solution of \eqref{eq:1.1}) and $\xi^* = (\xi^*_0, \xi^*_1, \ldots, \xi^*_n)^T$ is the exact solution of the perturbed system of nonlinear equations \eqref{eq:3.5}.
\begin{proposition}\label{pro2.3}
	Let the assumptions {\rm (i)--(v)} be satisfied. Then 
	\begin{equation}\label{eq:3.31}
	\|  x^* - \xi^* \| \le \frac{\sqrt{b-a} \;  | O(h^{\nu_1})| }{1-\alpha},
	\end{equation}
	where $\alpha = \frac{M^m}{ \sqrt{(m-1)!} } <1$ when $m$ is chosen sufficiently large.
\end{proposition}
\begin{proof} 
	By \eqref{eq:3.3} and \eqref{eq:3.4}, we have
	\begin{gather*}
	\begin{split} 
	& x^*_i - \xi^*_i +  \sum\limits_{j = 0}^{i }{w_{i_j} K_1(t_i, s_j, x^*_j)}    - \sum\limits_{j = 0}^{i }{w_{i_j} K_1(t_i, s_j, \xi^*_j)} +     \sum\limits_{r = 0}^{n }{w_r K_2(t_i, s_r, x^*_r)}\\
	&  - \sum\limits_{r = 0}^{n }{w_r K_2(t_i, s_r, \xi^*_r)}	= - O(h^{\nu_1}), \; i = 0, 1, \ldots, n, 
	\end{split}    
	\end{gather*}
	which means
	\[ x^*_i - \xi^*_i + \varphi_i(x^*) - \varphi_i(\xi^*) +  f_i(x^*) - f_i(\xi^*) 
	= - O(h^{\nu_1}),\; i = 0, 1, \ldots, n. \]
Then we have
	\[ |x^*_i - \xi^*_i + \varphi_i(x^*) - \varphi_i(\xi^*) +  f_i(x^*) - f_i(\xi^*) | 
	= | O(h^{\nu_1})|,\; i = 0, 1, \ldots, n,  \]
	and hence
	\[ w_i |x^*_i - \xi^*_i + \varphi_i(x^*) - \varphi_i(\xi^*) +  f_i(x^*) - f_i(\xi^*) |^2 
	= w_i | O(h^{\nu_1})|^2,\; i = 0, 1, \ldots, n.  \]	
It follows that
	\begin{gather*}
	\begin{split}
	& \|  x^* - \xi^*  + \Phi(x^*) - \Phi(\xi^*) + F(x^*) - F(\xi^*)\|^2 \\
	& = \sum\limits_{i = 0}^{n }{ w_i |x^*_i - \xi^*_i + \varphi_i(x^*) - \varphi_i(\xi^*) +  f_i(x^*) - f_i(\xi^*) |^2  }  
	= | O(h^{\nu_1})|^2	\sum\limits_{i = 0}^{n } {w_i }. 
	\end{split} 
	\end{gather*}
	Since in every Newton--Cotes formula $\sum\limits_{i = 0}^{n } {w_i } = b-a$, we obtain
	\[\|  x^* - \xi^* + \Phi(x^*) - \Phi(\xi^*) + F(x^*) - F(\xi^*) \| = \sqrt{b-a} \; | O(h^{\nu_1})|. \]	
	By virtue of the contraction of the mapping $T^m$ and the monotonicity of the mapping $F$, we have
	\begin{gather*}
	\begin{split}
	  \sqrt{b-a} \; | O(h^{\nu_1})| &  =  \|  x^* + F(x^*) - [ \xi^*  + F(\xi^*)] + \Phi(x^*) - \Phi(\xi^*)\| \\
	 & =     \|  z_x^* - z^* + \Phi P^{-1}(z_x^*) - \Phi P^{-1}(z^*) \|\\
	 & = \|  z_x^* - z^* + T(z_x^*) - T(z^*) \| \\
	 & =    \|  z_x^* - z^* + T(T^m(z_x^*)) - T(T^m(z^*)) \| \\
	 & = \|  z_x^* - z^* + T^{m+1}(z_x^*) - T^{m+1}(z^*) \|\\
	  & =   \|  z_x^* - z^* + T^m(T(z_x^*)) - T^m(T(z^*)) \| \\
	  &	\ge   \|  z_x^* - z^*\|  - \| T^m(T(z_x^*)) - T^m(T(z^*)) \| \\
	  & \ge   \|  z_x^* - z^*\| - \alpha \|  T(z_x^*) - T(z^*)\| \\
	 & =  (1- \alpha) \|  z_x^* - z^*\| \\
	 & =   (1- \alpha)  \|  x^* - \xi^* + F(x^*) - F(\xi^*) \| \\
	 &  \ge   (1- \alpha) \|  x^* - \xi^* \|,
	\end{split} 
	\end{gather*}
	where  $z_x^* = x^* + F(x^*) \equiv P(x^*)$ and $z^* = \xi^* + F(\xi^*) \equiv P(\xi^*)$.
	Consequently,
	\[\|  x^* - \xi^* \| \le \frac{\sqrt{b-a} \;  | O(h^{\nu_1})| }{1-\alpha}. \]	
	This completes the proof of the proposition.
\end{proof} 
\noindent  The inequality \eqref{eq:3.31} leads to the following corollary.
\begin{corollary}\label{cro2.1}
	$\|  x^* - \xi^* \|$ vanishes when $h \to 0$.
\end{corollary}
\indent Next, we shall construct the iterative algorithm to find approximate solutions of the perturbed system of nonlinear equations \eqref{eq:3.5}. To solve the perturbed system of nonlinear  equations \eqref{eq:3.5}, we first have to solve the system of  equations \eqref{eq:3.28} and after that we solve  the system of  equations \eqref{eq:3.27}. In the proof of Theorem \ref{thm3.3}, we have proved that the system of  equations \eqref{eq:3.28} has a unique solution by using the contraction mapping principle and the system of  equations \eqref{eq:3.27} has a unique solution for any $z \in \mathbb{R}^{n+1}$ by using parameter continuation method. The approximate solutions of the system of  equations \eqref{eq:3.28} are obtained by using the standard iteration process 
\begin{equation}\label{eq:3.32}
z^{(\tau+1)} =    - \Phi P^{-1}(z^{(\tau)}) + g \equiv T(z^{(\tau)} ), \; \tau =0, 1, 2, \ldots. 
\end{equation}
For the initial approximation we take $z^{(0)} = g$. At the same time at each step of above iteration process when calculating the value $P^{-1}(z^{(\tau)})$,  we have to solve the system of equations  of the form \eqref{eq:3.27}, as
\begin{equation}\label{eq:3.33}
\xi + F(\xi) = z^{(\tau)}.
\end{equation}
Substituting $F$ for $A$ in the iteration processes \eqref{eq:2.5a}--\eqref{eq:2.5d},  the approximate solutions of  the  system of equations \eqref{eq:3.33} are obtained by using the following iteration processes 
\begin{subequations}
	\begin{align}
	\xi^{(k+1)} & = - \varepsilon_0 F(\xi^{(k)}) + u^{(l)},\; k=0, 1, 2, \ldots,\label{eq:3.34a}\\
	u^{(l+1)} & = - \varepsilon_0 F G_1^{-1}(u^{(l)}) + v^{(c)},\; l=0,1, 2,  \ldots ,\label{eq:3.34b}\\
	& \; \ldots,\label{eq:3.34c}\\
	y^{(p+1)} & = - \varepsilon_0 F G_1^{-1} \cdots G_{N-1}^{-1}(y^{(p)}) + z^{(\tau)}, \; p=0, 1, 2, \ldots.\label{eq:3.34d}
	\end{align}
\end{subequations}
Therefore the approximate  solutions of the perturbed system of nonlinear equations \eqref{eq:3.5}  can be found by the following iteration processes
\begin{subequations}
	\begin{align}
	\xi^{(k+1)} & = - \varepsilon_0 F(\xi^{(k)}) + u^{(l)},\; k=0, 1, 2, \ldots,\label{eq:3.35a}\\
	u^{(l+1)} & = - \varepsilon_0 F G_1^{-1}(u^{(l)}) + v^{(c)},\; l=0,1, 2,  \ldots ,\label{eq:3.35b}\\
	& \; \ldots,\label{eq:3.35c}\\
	y^{(p+1)} & = - \varepsilon_0 F G_1^{-1} \cdots G_{N-1}^{-1}(y^{(p)}) + z^{(\tau)}, \; p = 0, 1, 2, \ldots\;\label{eq:3.35d}\\
	z^{(\tau+1)} & = - \Phi P^{-1}(z^{(\tau)}) + g, \; \tau = 0, 1, 2, \ldots, z^{(0)} = g. \label{eq:3.35e}
	\end{align}
\end{subequations}
\indent Now  we estimate the error of approximate solutions of the  perturbed system of nonlinear equations \eqref{eq:3.5}. Assume that the number of steps in each iteration scheme of the iteration processes \eqref{eq:3.35a}--\eqref{eq:3.35e} is the same and equals $n_0$. Let $\xi^{(n_0)}$ be approximate solutions of the perturbed system of nonlinear  equations  \eqref{eq:3.5}. Note that $\xi^{(n_0)}$ depends on $N$, hence we denote $\xi(n_0,N) \equiv \xi^{(n_0)}$. We have the following result. 
\begin{theorem}\label{thm3.4}
	Let the assumptions of Theorem  \ref{thm3.3} be satisfied. Then the sequence of approximate solutions $ \{\xi(n_0,N)\}, \;n_0 = 1, 2, \ldots$  constructed by iteration processes \eqref{eq:3.35a}--\eqref{eq:3.35e} converges to the exact solution $\xi^*$ of  the perturbed system of nonlinear equations  \eqref{eq:3.5}. Moreover, the following estimates hold
	\begin{equation}\label{eq:3.36}
	\|\xi(n_0,N) - \xi^* \| \le   (1+M)  \frac{q^{n_0+1}}{1-q} \frac{e^{qN}-1}{e^q - 1}  C_{n'} \|g\| +    \frac{\alpha^{\frac{n_0-h_0}{m}}}{1-\alpha} C_m  \|g\|, 
	\end{equation} 
	where 
	$	C_{n'}  =  \frac{\gamma}{1-\gamma}  \frac{M^{n'} }{\sqrt{(n' -1)!}} +   \frac{M^{n'} }{\sqrt{(n'-1)!}} +  \cdots + M+1,
	C_m 	 = \frac{M^m}{\sqrt{(m-1)!}}  + \frac{M^{m-1}}{\sqrt{(m-2)!}} + \cdots + M,$ 
	$N$ is the smallest natural number such that $q = \frac{L}{N}<1$, $n'$ is a natural number such that $\gamma = \frac{M}{\sqrt{n'}} < 1$ and $m$ is chosen sufficiently large such that $\alpha = \frac{M^m}{\sqrt{(m-1)!}} < 1$, $h_0 \in \left\lbrace 0, 1, \ldots, m-1 \right\rbrace $ is the residual of $\frac{n_0}{m}$, $n_0 >	max  \left\{m, n'  \right\} $.
\end{theorem}
\begin{proof}
	For simplicity, we assume that  $\Phi(0) = 0$ and $F(0) = 0 $, where $0 = (0, 0 , \ldots,0)^T$ denotes the zero element in $\mathbb{R}^{n+1}$. Indeed, if $\Phi(0) \ne 0$  or $F(0) \ne 0$, we can define two mappings $\Phi_1, F_1: \mathbb{R}^{n+1} \to \mathbb{R}^{n+1}$ by 
	\begin{equation}\label{eq:3.37}
	\Phi_1(\xi) = \Phi(\xi) - \Phi(0), F_1(\xi) = F(\xi) - F(0), 	
	\end{equation} 
	then $\Phi_1(0) = F_1(0) =0$ and the perturbed system of nonlinear equations  \eqref{eq:3.5} is equivalent to
	\begin{equation}\label{eq:3.38}
	\xi  + \Phi_1(\xi) + F_1(\xi)= g_1, 	
	\end{equation} 
	where $g_1 =  g  - \Phi(0)- F(0)$. It follows from \eqref{eq:3.37} that for all $\xi, \overline{\xi} \in \mathbb{R}^{n+1}$ 
	\[\Phi_1(\xi) - \Phi_1(\overline{\xi}) = \Phi(\xi) - \Phi(\overline{\xi}), F_1(\xi) - F_1(\overline{\xi}) = F(\xi) - F(\overline{\xi}).\]
	Therefore  the Propositions \ref{pro2.1} and \ref{pro2.2} can be applied to the mappings $\Phi_1$ and $F_1$, respectively.  Consequently, Theorem \ref{thm3.3} can be applied to the perturbed system of nonlinear equations \eqref{eq:3.38}. \\
	\indent We split the proof into two steps.\\
	\indent Step $1$. We estimate the error of approximate solutions of the system of  equations \eqref{eq:3.28}. Firstly,  we estimate the errors in  calculating the values $T(z^{(\tau)} ) =  - \Phi P^{-1}(z^{(\tau)}) + g, \; \tau = 1, 2, \ldots, n_0 -1$. Since $F(0) = 0$, it follows that $P(0) = 0 + F(0) =0$. Thus $T(0) = - \Phi P^{-1}(0) + g = g \equiv z^{(0)}$. At the same time at each step of the iteration process \eqref{eq:3.32} when calculating the value $P^{-1}(z^{(\tau)})$,  we will use the iteration processes \eqref{eq:3.34a}--\eqref{eq:3.34d}. Let $\xi^{(n_0)}_{z^{(\tau)}}$ and $\xi^*_{z^{(\tau)}}$ be the approximate and exact values of $P^{-1}(z^{(\tau)})$, respectively. It follows from Theorem \ref{thm2.2} that the values $P^{-1}(z^{(\tau)})$ are calculated with the error
	\begin{equation}\label{eq:3.39}
	\|\xi^{(n_0)}_{z^{(\tau)}}- \xi^*_{z^{(\tau)}} \| \le \frac{q^{n_0+1}}{1-q} \frac{e^{qN}-1}{e^q - 1} \|z^{(\tau)}\|,
	\end{equation}
	where $N$ is the smallest natural number such that $q = \frac{L}{N}<1, n_0 = 1, 2, \ldots$. \\
	\indent Since $\left\lbrace z^{(\tau)} \right\rbrace,\; \tau =1, 2, \ldots$ is a convergence sequence, it follows that  $\|z^{(\tau)}\|$ is bounded for all positive integer $\tau$.  We now determine the supremum  of $\|z^{(\tau)}\|, \;\tau \in \left\lbrace 1, 2, \ldots, n_0 \right\rbrace$ ($n_0$ is the number of steps in each iteration scheme). For any $\tau \in \left\lbrace 1, 2, \ldots, n_0 \right\rbrace $ we have
	\begin{gather}\label{eq:3.40} 
	\begin{split} 
	\|z^{(\tau)} \| &  \le \|z^{(\tau)} - z^{(\tau-1)}\| + \|z^{(\tau-1)} - z^{(\tau-2)}\|+ \cdots + \|z^{(1)} - z^{(0)}\|  + \|z^{(0)}\|  \\
	& \le \|z^{(n_0)} - z^{(n_0-1)}\| + \|z^{(n_0-1)} - z^{(n_0-2)}\|+ \cdots + \|z^{(\tau+1)} - z^{(\tau)}\|   \\
	& \;\;\; + \|z^{(\tau)} - z^{(\tau-1)}\| + \|z^{(\tau-1)} - z^{(\tau-2)}\|+ \cdots + \|z^{(1)} - z^{(0)}\|  + \|z^{(0)}\|.
	\end{split}
	\end{gather}
	On the other hand, we have
	\begin{align*}
	& \|z^{(n_0)} - z^{(n_0-1)}\| + \|z^{(n_0-1)} - z^{(n_0-2)}\|+ \cdots +  \|z^{(1)} - z^{(0)}\|  + \|z^{(0)}\|\\
	= &  \|T^{n_0}(g) - T^{n_0}(0)\| + \|T^{n_0-1}(g) - T^{n_0 -1}(0)\| + \cdots + \|T(g) - T(0) \| + \|g \|\\
	\le & \left[ \frac{M^{n_0} }{\sqrt{(n_0-1)!}} + \frac{M^{n_0 -1} }{\sqrt{(n_0-2)!}} + \cdots +  M  \right] \|g\| + \|g\|\\
	= &  \left[ \frac{M^{n_0} }{\sqrt{(n_0 -1)!}} + \frac{M^{n_0 -1} }{\sqrt{(n_0-2)!}} + \cdots +  M +1 \right] \|g\|.
	\end{align*}
	Let $n'$ be natural number such that $\gamma = \frac{M}{\sqrt{n'}} < 1$. Then for any $n_0 > n'$, we have
	\begin{gather*}
	\begin{split}
	& \frac{M^{n_0} }{\sqrt{(n_0 -1)!}} + \frac{M^{n_0 -1} }{\sqrt{(n_0-2)!}} + \cdots +  M +1\\
	= &  \frac{M^{n_0} }{\sqrt{(n_0-1)!}}  + \frac{M^{n_0-1} }{\sqrt{(n_0 -2)!}}  + \cdots + \frac{M^{n' +1} }{\sqrt{n'!}} + \frac{M^{n'} }{\sqrt{(n'-1)!}} +  \cdots + M +1\\
	\le &  \frac{M^{n'} }{\sqrt{(n'-1)!}} \gamma^{n_0 - n'} + \frac{M^{n'} }{\sqrt{(n'-1)!}} \gamma^{n_0 - n'-1} + \cdots + \frac{M^{n'} }{\sqrt{(n'-1)!}} \gamma + \frac{M^{n'} }{\sqrt{(n'-1)!}} \\
	& + \cdots  + M+1\\
	= &  ( \gamma^{n_0 - n'}+ \gamma^{n_0 - n' -1} +  \cdots + \gamma  )  \frac{M^{n'} }{\sqrt{(n'-1)!}}  + \frac{M^{n'} }{\sqrt{(n'-1)!}} +  \cdots + M +1\\ 
	= &   \gamma \frac{1- \gamma^{n_0 - n'}}{1- \gamma} \frac{M^{n'} }{\sqrt{(n'-1)!}} +   \frac{M^{n'} }{\sqrt{(n' -1)!}} + \cdots + M + 1\\
	\le &    \frac{\gamma}{1-\gamma}  \frac{M^{n'} }{\sqrt{(n'-1)!}} +   \frac{M^{n'} }{\sqrt{(n'-1)!}} + \cdots  + M +1 \equiv C_{n'}.
	\end{split} 
	\end{gather*} 
	Thus
	\begin{equation}\label{eq:3.41}
	\|z^{(n_0)} - z^{(n_0-1)}\| + \|z^{(n_0-1)} - z^{(n_0-2)}\|+ \cdots +  \|z^{(1)} - z^{(0)}\|  + \|z^{(0)}\|  \le     C_{n'} \|g\|. 
	\end{equation}
	It follows from \eqref{eq:3.40} and \eqref{eq:3.41}  that
	\begin{equation*}
	\|z^{(\tau)} \| \le     C_{n'} \|g\|
	\end{equation*}
	for any $\tau \in \left\lbrace 1, 2, \ldots, n_0 \right\rbrace $. Hence the values $P^{-1}(z^{(\tau)})$ are calculated with the error
	\[ \|\xi^{(n_0)}_{z^{(\tau)}}- \xi^*_{z^{(\tau)}} \|  \le  \Delta(n_0),      \]
	where
	\begin{equation}\label{eq:3.42}
	\Delta(n_0) =  \frac{q^{n_0 + 1}}{1-q} \frac{e^{qN}-1}{e^q - 1}  C_{n'} \|g\|.   
	\end{equation}
	By Proposition \ref{pro2.1}, we have
	\[ \| T(z) - T(\overline{z}) \| = \| \Phi P^{-1}(z) - \Phi P^{-1}(\overline{z}) \| \le  M \| P^{-1}(z) -P^{-1}(\overline{z})\|,\; \forall  z, \overline{z} \in \mathbb{R}^{n+1}.\]   
	Therefore  the values $T(z^{(\tau)} ) =  - \Phi P^{-1}(z^{(\tau)}) + g, \;\tau = 1, 2, \ldots, n_0 -1$ are calculated with the error not more than $M \Delta(n_0)$.\\
	\indent Next, we shall  estimate the error of an iteration process in the calculation of $z$. 	By Theorem \ref{thm3.2}, the error of an iteration process in the calculation of $z$ equals $  \frac{\alpha^{\frac{n_0 - h_0}{m}}}{1-\alpha} \| z^{(m+h_0)}-z^{(h_0)} \|$, 
	where  $m$ is chosen sufficiently large such that $\alpha = \frac{M^m}{\sqrt{(m-1)!}}$ $< 1$, $h_0 \in \left\lbrace 0, 1, \ldots, m-1 \right\rbrace $ is the residual of $\frac{n_0}{m}$. We have
	\begin{gather}\label{eq:3.43}
	\begin{split}
	& \| z^{(m+h_0)}-z^{(h_0)} \|  \\ 
	\le & \| z^{(m+h_0)}-z^{(m+ h_0 - 1)} \|
	+ \| z^{(m+h_0-1)}-z^{(m+ h_0 - 2)} \| 
	+ \cdots    + \| z^{(h_0+1)}-z^{( h_0 )} \|  \\
	= & \| T^{m+h_0}(g) - T^{m+h_0}(0) \| +\| T^{m+h_0-1}(g) - T^{m+h_0-1}(0) \| 	 \\
	&  + \cdots+   \| T^{h_0+1}(g)- T^{h_0+1}(0) \|  \\
	\le  & 	  \left[ \frac{M^{m+h_0}}{\sqrt{(m+h_0-1)!}} 
	+ \frac{M^{m+h_0-1}}{\sqrt{(m+h_0-2)!}}
	+ \cdots + \frac{M^{h_0+1}}{\sqrt{h_0!}} \right]  \|g\|  \equiv R_{m,h_0} \|g\|,   
		\end{split}
	\end{gather}
	where $R_{m,h_0} = \frac{M^{m+h_0}}{\sqrt{(m+h_0-1)!}}  + \frac{M^{m+h_0-1}}{\sqrt{(m+h_0-2)!}} + \cdots + \frac{M^{h_0+1}}{\sqrt{h_0!}}$. 
	Let us prove that
	\begin{align}  
	\max \{R_{m,h_0}, h_0 \in \left\lbrace 0, 1, \ldots, m-1 \right\rbrace\}  & = \frac{M^m}{\sqrt{(m-1)!}}  + \frac{M^{m-1}}{\sqrt{(m-2)!}} + \cdots + M \nonumber \\
	& 	\equiv C_m.  \label{eq:3.44}
    \end{align} 
	Obviously, $R_{m,0} = C_m$. For some non-negative integer $h', 0 \le h' \le h_0$,  we have
	\[R_{m,h'} =  \frac{M^{m+h'}}{\sqrt{(m+h'-1)!}}  + \frac{M^{m+h'-1}}{\sqrt{(m+h'-2)!}} + \cdots + \frac{M^{h'+1}}{\sqrt{h'!}}          \]
	and
	\[ R_{m,h'+1}  =  \frac{M^{m+h'+1}}{\sqrt{(m+h')!}} + \frac{M^{m+h'}}{\sqrt{(m+h'-1)!}}  + \frac{M^{m+h'-1}}{\sqrt{(m+h'-2)!}}  + \cdots + \frac{M^{h'+2}}{\sqrt{(h'+1)!}}.\]
	Hence
	\[R_{m,h'+1} - R_{m,h'} =  \frac{M^{m+h'+1}}{\sqrt{(m+h')!}} -  \frac{M^{h'+1}}{\sqrt{h'!}}.   \]
	We have 
	\begin{align*}
	\frac{M^{m+h'+1}}{\sqrt{(m+h')!}} &  = \frac{M^m}{\sqrt{(m-1)!}}  \frac{M^{h'+1}}{\sqrt{m(m+1) \cdots (m+h')}}\\
	&   = \alpha \frac{M^{h'+1}}{\sqrt{m(m+1) \cdots (m+h')}} .
	\end{align*}
	Since $\alpha \in \left[ 0,1\right) $, it follows that
	\begin{equation}\label{eq:3.45}
	\frac{M^{m+h'+1}}{\sqrt{(m+h')!}} \le \frac{M^{h'+1}}{\sqrt{m(m+1) \cdots (m+h')}}. 
	\end{equation}
	On the other hand,  we have 
	\[m(m+1) \cdots (m+h') \ge 1 (1+1) \cdots (h' +1) \ge h'! \]
	for any positive integer $m$. Hence
	\begin{equation}\label{eq:3.46}
	\frac{M^{h' +1}}{\sqrt{m(m+1) \cdots (m+h')}} \le  \frac{M^{h'+1}}{\sqrt{h'!}}.
	\end{equation}
	Combining  \eqref{eq:3.45} and  \eqref{eq:3.46}, we get
	\[\frac{M^{m+h'+1}}{\sqrt{(m+h')!}} \le \frac{M^{h'+1}}{\sqrt{h'!}}. \]
	Thus
	\[  R_{m,h'+1} - R_{m,h'} \le 0,         \]
	which implies that $R_{m,h'+1} \le  R_{m,h'}$ for some non-negative integer $h', 0 \le h' \le h_0$. Therefore \eqref{eq:3.44} is proved. 
	It follows from  \eqref{eq:3.43} and \eqref{eq:3.44}  that
	\[  \| z^{(m+h_0)}-z^{(h_0)} \|  \le C_m  \|g\|   \]
	for every integer $h_0 \in \left\lbrace 0, 1, \ldots, m-1 \right\rbrace $. Hence the error of an iteration process in the calculation of $z$ equals 
	$ \frac{\alpha^{\frac{n_0-h_0}{m}}}{1-\alpha} C_m  \|g\|$.\\ 
	\indent Consequently, the error of approximate solutions $z^{(n_0)}$ of the system of equations \eqref{eq:3.28} gives the estimate
	\[     \| z^{(n_0)} - z^*\| \le M \Delta(n_0) + \frac{\alpha^{\frac{n_0-h_0}{m}}}{1-\alpha} C_m  \|g\|.  \]
	\indent Step $2$. We estimate the error of approximate solutions of the  system of equations \eqref{eq:3.27}
	\[P(\xi) \equiv \xi + F(\xi) = z.\]
	Since the mapping $P^{-1}$ is Lipschitz - continuous with Lipschitz coefficient equal to $1$, the substitution of the error $ M \Delta(n_0) + \frac{\alpha^{\frac{n_0-h_0}{m}}}{1-\alpha} C_m \|g\|$ into the right - hand side of the system of equations \eqref{eq:3.27} causes an error of not more than $ M \Delta(n_0) + \frac{\alpha^{\frac{n_0-h_0}{m}}}{1-\alpha} C_m  \|g\|$ in the corresponding solution $\xi$. The error of an iteration process in the calculation of $\xi$ equals $\Delta(n_0)$. Consequently,
	\begin{gather*}
	\begin{split}
	\| \xi(n_0, N) - \xi^*\|  & \le M \Delta(n_0) +  \frac{\alpha^{\frac{n_0-h_0}{m}}}{1-\alpha} C_m  \|g\| + \Delta(n_0) \\
	&  = (1+M)  \Delta(n_0) +  \frac{\alpha^{\frac{n_0-h_0}{m}}}{1-\alpha} C_m  \|g\|.   	
	\end{split} 
	\end{gather*} 
	By \eqref{eq:3.42}, we obtain  \eqref{eq:3.36}.
	This completes the proof of the theorem.
\end{proof}
\begin{remark}\label{rem:3.1}
	Let $d$ be integer part of $\frac{n_0}{m}$, i.e., $n_0 = m d + h_0,$ $ h_0 \in \left\lbrace 0, 1, \ldots, m-1 \right\rbrace$. We have $n_0 + 1 = m d +h_0+1 \ge m d + 1$ for every  $h_0 \in \left\lbrace 0, 1, \ldots, m-1 \right\rbrace$.  Since $0 < q < 1$,  it follows that $q^{n_0 + 1} \le q^{m d+1}$. From this and  \eqref{eq:3.36},  we have 
	\begin{gather}\label{eq:3.47}
	\begin{split}
	\!{\| \xi(n_0,N) - \xi^*\|} & \le  (1+M)  \frac{q^{m d+1}}{1-q} \frac{e^{qN}-1}{e^q - 1}  C_{n'} \|g\| +    \frac{\alpha^{\frac{n_0-h_0}{m}}}{1-\alpha} C_m  \|g\|\\
	& =  (1+M) \frac{q}{1-q} \frac{e^{qN}-1}{e^q - 1}  C_{n'} \|g\| q^{m d}   +    \frac{ C_m }{1-\alpha}   \|g\| \alpha^{d}\\
	& = C_1 q^{m d} + C_2 \alpha^{d},
	\end{split} 
	\end{gather} 
	where 
	\[ C_1 =    (1+M)\frac{q}{1-q}  \frac{e^{qN}-1}{e^q - 1}  C_{n'}  \| g \|, C_2 = \frac{C_m }{1-\alpha}\| g \|. \]
	Let $ \beta= \max\left\{ {q^m, \alpha} \right\}$.  From \eqref{eq:3.47}, we get
	\begin{equation}\label{eq:3.48}
	\|\xi(n_0,N) - \xi^* \|  \le (C_1 +  C_2 ) \beta^{d}.     
	\end{equation}
	It follows from \eqref{eq:3.48} that for any given $\epsilon > 0$, we can find the number of iteration such that $\|\xi(n_0,N) - \xi^* \| \le \epsilon $.
\end{remark}	
\begin{remark}\label{rem:3.2}
	We shall now estimate the complexity of the proposed iterative algorithm \eqref{eq:3.35a}--\eqref{eq:3.35e}. The iteration processes \eqref{eq:3.35a}--\eqref{eq:3.35e} can be written as the following symbolic notation
	\begin{gather}\label{eq:3.49}
	\begin{split}	
	\xi^{(k+1)} &  = \underbrace {-  \frac{1}{N} F(\xi^{(k)}) - \frac{1}{N}F(\xi^{(l)}) - \cdots -  \frac{1}{N} F(\xi^{(p)})}_{N\;terms} - \Phi(\xi^{(\tau)})  + g,\\
	& \;\;\;\; k, l, \ldots,  \tau = 0, 1, \ldots, n_0.
	\end{split}	
	\end{gather}
	The procedure for calculating each value $F(\xi), \Phi(\xi)$ in the specified element $\xi$ is called an elementary operation. We shall call the number of elementary operations necessary to implement algorithm \eqref{eq:3.35a}--\eqref{eq:3.35e} is the volume of the calculations $M(n_0,N)$. From the symbolic notation \eqref{eq:3.49} it follows that    $   M(n_0,N) \le  (n_0+1)^{N+1}$.
\end{remark}
\section{Illustrative example}
\label{sec:4}
In this section, to illustrate our above results an example is presented. The computations associated with the example were performed using Maple 12
on personal computer.
\begin{example}\label{ex:1}
	Consider the nonlinear Volterra--Fredholm integral equation
	\begin{gather}\label{eq:4.1}
	\begin{split}
	& x(t)  + 5 \int\limits_0^{t}{ts \cos[x(s)]ds} + \frac{11}{2} \int\limits_0^1{t^2 s^2 \, x(s)ds} =\frac{11}{8}t^2 - 4t+ 5 t \cos(t) \\
	&  + 5 t^2  \sin(t),\;  0 \le  t \le 1.  
	\end{split}
	\end{gather}
\end{example}
\noindent The analytical solution of this integral equation is $x(t)=t$ on $[0,1]$. It is easy to verify that the functions $K_1(t,s,x) = 5 ts  \cos[x(s)]$ and  $K_2(t,s,x) = \frac{11}{2} t^2 s^2 \,x(s)$ satisfy the conditions (i)-(v)  of the Theorem \ref{thm3.3} with $M^2 = \frac{25}{18}, L^2 = \frac{121}{100}$. For approximating the left-hand integrals,  we use composite midpoint rule and take a partition with the discretization parameter $h = \frac{1}{50}$.  We take $\epsilon = 10^{-3}, N=2, m=8, n' = 6$.  It follows from \eqref{eq:3.48} that $d \ge 6$. Taking $d = 6$. Since $n_0 = m d + h_0 = 8 d + h_0 ,\; h_0 \in \left\lbrace 0, 1, \ldots, 7 \right\rbrace$,  we have $n_0 \ge 55 $. Taking $n_0 = 55$,  the  number of iterations needed is $166375$  (the number of steps in each iteration scheme is the same and equals $n_0 = 55$).    Table \ref{tab:1} presents approximate solutions obtained by using the iteration processes \eqref{eq:3.35a}--\eqref{eq:3.35e} with $N=2$ and $n_0 = 55$, also exact solutions are given for comparison.   
\begin{table}[tbhp]
	\caption{ Comparison of the exact and approximate solutions for Example \ref{ex:1}.}
		\label{tab:1}
	\centering
	\scriptsize {\tabcolsep = 0.2mm
		\begin{tabular}{cccccccc}
			\hline\noalign{\smallskip}
			Nodes & Exact   & Approximate  & Absolute & \; Nodes  & Exact   & Approximate  & Absolute \\
			$t$ & solutions & solutions & error & $\;t$ & solutions & solutions & error\\
			\noalign{\smallskip}\hline\noalign{\smallskip}
			$0.01$ & $0.01$ & $0.0099948368$ & $\;\; 5.1632000 \times 10^{-6}$& $0.51$ & $0.51$  & $0.5041937466$ & $\;\; 5.80625340 \times 10^{-3}$  \\
			$0.03$ & $0.03$ & $0.0299685542$ & $\;\;3.1445800 \times 10^{-5}$& $0.53$ & $0.53$  & $0.5238057793$ & $\;\;6.19422070 \times 10^{-3}$  \\
			$0.05$ & $0.05$ & $0.0499210957$ & $\;\;7.8904300 \times 10^{-5}$& $0.55$ & $0.55$  & $0.5434091045$ & $ \;\;6.59089550\times 10^{-3}$  \\
			$0.07$ & $0.07$ & $0.0698526617$ & $\;\;1.4733830 \times 10^{-4}$& $0.57$ & $0.57$  & $0.5630033477$ & $\;\; 6.99665230\times 10^{-3}$  \\
			$0.09$ & $0.09$ & $0.0897635423$ & $\;\;2.3645770 \times 10^{-4}$& $0.59$ & $0.59$  & $0.5825879388$ & $ \;\;7.41206120\times 10^{-3}$  \\
			$0.11$ & $0.11$ & $0.1096541138$ & $\;\;3.4588620 \times 10^{-4}$& $0.61$ & $0.61$  & $0.6021621161$ & $ \;\;7.83788390\times 10^{-3}$  \\
			$0.13$ & $0.13$ & $0.1295248320$ & $\;\;4.7516800 \times 10^{-4}$& $0.63$ & $0.63$  & $0.6217248525$ & $\;\;8.27514750 \times 10^{-3}$  \\
			$0.15$ & $0.15$ & $0.1493762257$ & $\;\;6.2377430 \times 10^{-4}$& $0.65$ & $0.65$  & $0.6412748634$ & $ \;\;8.72513660\times 10^{-3}$  \\
			$0.17$ & $0.17$ & $0.1692088902$ & $\;7.9110980 \times 10^{-4}$& $0.67$ & $0.67$  & $0.6608105470$ & $\;\; 9.18945300\times 10^{-3}$  \\
			$0.19$ & $0.19$ & $0.1890234780$ & $\;\;9.7652200 \times 10^{-4}$& $0.69$ & $0.69$  & $0.6803299576$ & $\;\; 9.67004240\times 10^{-3}$  \\
			$0.21$ & $0.21$ & $0.2088206886$ & $\;\;1.1793114 \times 10^{-3}$& $0.71$ & $0.71$  & $0.6998307414$ & $\;\; 1.01692586\times 10^{-2}$  \\
			$0.23$ & $0.23$ & $0.2286012617$ & $\;\;1.3987383 \times 10^{-3}$& $0.73$ & $0.73$  & $0.7193100805$ & $\;\; 1.06899195\times 10^{-2}$  \\
			$0.25$ & $0.25$ & $0.2483659629$ & $\;\;1.6340371 \times 10^{-3}$& $0.75$ & $0.75$  & $0.7387646198$ & $\;\; 1.12353802\times 10^{-2}$  \\
			$0.27$ & $0.27$ & $0.2681155773$ & $\;\;1.8844227 \times 10^{-3}$& $0.77$ & $0.77$  & $0.7581903691$ & $\; 1.18096309\times 10^{-2}$  \\
			$0.29$ & $0.29$ & $0.2878508938$ & $\;\;2.1491062 \times 10^{-3}$& $0.79$ & $0.79$  & $0.7775825937$ & $\; 1.24174063\times 10^{-2}$  \\
			$0.31$ & $0.31$ & $0.3075726969$ & $\;\;2.4273031 \times 10^{-3}$ & $0.81$ & $0.81$  & $0.7969356971$ & $\;\;1.30643029 \times 10^{-2}$  \\
			$0.33$ & $0.33$ & $0.3272817556$ & $\;\;2.7182444 \times 10^{-3}$ & $0.83$ & $0.83$  & $0.8162430208$ & $\;\;1.37569792 \times 10^{-2}$  \\
			$0.35$ & $0.35$ & $0.3469788114$ & $\;\;3.0211886 \times 10^{-3}$ & $0.85$ & $0.85$  & $0.8354966758$ & $\;\; 1.45033242\times 10^{-2}$  \\
			$0.37$ & $0.37$ & $0.3666645658$ & $\;\;3.3354342 \times 10^{-3}$ & $0.87$ & $0.87$  & $0.8546872738$ & $\;\;1.53127262 \times 10^{-2}$  \\
			$0.39$ & $0.39$ & $0.3863396700$ & $\;\;3.6603300 \times 10^{-3}$ & $0.89$ & $0.89$  & $0.8738036297$ & $\;\; 1.61963703\times 10^{-2}$  \\
			$0.41$ & $0.41$ & $0.4060047106$ & $\;\;3.9952894 \times 10^{-3}$ & $0.91$ & $0.91$  & $0.8928323854$ & $\;\; 1.71676146\times 10^{-2}$  \\
			$0.43$ & $0.43$ & $0.4256602069$ & $\;\;4.3397931 \times 10^{-3}$ & $0.93$ & $0.93$  & $0.9117575067$ & $\;\;1.82424933 \times 10^{-2}$  \\
			$0.45$ & $0.45$ & $0.4453065831$ & $\;\;4.6934169 \times 10^{-3}$ & $0.95$ & $0.95$  & $0.9305597280$ & $\;\; 1.94402720\times 10^{-2}$  \\
			$0.47$ & $0.47$ & $0.4649441741$ & $\;\;5.0558259 \times 10^{-3}$ & $0.97$ & $0.97$  & $0.9492157979$ & $\;\; 2.07842021\times 10^{-2}$  \\
			$0.49$ & $0.49$ & $0.4845731957$ & $\;\;5.4268043 \times 10^{-3}$ & $0.99$ & $0.99$  & $0.9676975410$ & $\;\; 2.23024590\times 10^{-2}$  \\
			\noalign{\smallskip}\hline
	\end{tabular}}
\end{table}
\section{Conclusions}
\label{sec:5}
In this paper, a new numerical method has been proposed to solve nonlinear Volterra--Fredholm integral equations.  With this method, the perturbed system of nonlinear equations obtained by discretization is solved  by an iterative method, which is based on a hybrid of the method of contractive mapping and parameter continuation method. Lastly, an illustrative example is given to demonstrate  the effectiveness and convenience of the proposed method.

\end{document}


\maketitle
 \vspace*{-0.8truecm}

\begin{center}
July 15, 2019
\end{center}

 \vspace*{0.3truecm}

\noindent {\bf Abstract.}  Some corrections are made in our article,    which was published in {\it  Appl. Anal. Optim.} Vol. 3 (2019), No. 1, 103--127.
These  corrections  are intended to  transform the  equation \eqref{eq:1.1}  
\begin{equation}\label{eq:1.1}
x(t) +  \int\limits_a^t {K_1(t,s,x(s)) ds} + \int\limits_a^b {K_2(t,s,x(s)) ds} = g(t),\;\,a \le t \le b \tag{1.1}
\end{equation}
into a discretized form in a tighter and more accurate way without affecting the main results of the article.

 \vspace*{0.3truecm}

\noindent {\bf Key Words.} Volterra--Fredholm integral equations, quadrature method, parameter continuation method, discretization, approximation.

 \vspace*{0.3truecm}
 
\noindent {\bf 2010 Mathematics Subject Classification.} 45L05, 45G10,  47J25.
 \vspace*{0.3truecm}
 
The following items are the corrections to be made to the article  \symbol{92}A numerical method for solving nonlinear Volterra--Fredholm integral equations\symbol{34}.

\begin{itemize}
	\item [\rm 1.]  We find out that the transformation of equation  \eqref{eq:1.1}  into a discretized form   by using Collocation method and Quadratue method in the space   $L^2[a,b]$ is not  tight enough. Therefore  
	\symbol{92}$L^2[a,b]$\symbol{34} should be changed  into  \symbol{92}$C^{\nu}[a,b]$\symbol{34}  throughout  section 3. Besides, 
	we need to  add the following assumptions 
	to those introduced in  section 3:
		\item [-] Supposing that the 
		integral equation  \eqref{eq:1.1}  has   solution.
		\item [-] Participating functions in the equation  \eqref{eq:1.1} must have some degree of smoothness. For details,
		\begin{itemize}
	\item[\rm (i)] $g(t) \in C^{\nu}[a,b], \nu \ge 2$;
	\item[\rm (ii)]  $K_1(t,s,x(s)), K_2(t,s,x(s))$ are differentiable continuous functions up to 	 order $\nu $ on $\Omega = [a,b] \times [a,b]\times \mathbb R $, where $\nu \ge 2$.
		
\end{itemize}
\item [\rm 2.]  In section  4,   Example 4.1 should be removed because it does not match the changes in section 3.

\end{itemize}
